\documentclass[11 pt, reqno]{amsart}

\allowdisplaybreaks[1]

\usepackage{relsize}
\usepackage{amsmath,amsfonts,amssymb,amsthm,mathrsfs}
\usepackage{hyperref}
\usepackage{cleveref}
\usepackage[margin=2.5cm]{geometry}
\usepackage{setspace}
\usepackage{mathtools}

\usepackage{setspace}
\usepackage{xcolor}


\allowdisplaybreaks[1]

\numberwithin{equation}{section}
\newtheorem{theorem}{\bf Theorem}[section]
\newtheorem{lemma}[theorem]{\bf Lemma}

\newtheorem{prop}[theorem]{\bf Proposition}
\newtheorem{defn}[theorem]{\bf Definition}
\newtheorem{example}[theorem]{\bf Example}
\newtheorem{corollary}[theorem]{\bf Corollary}

\makeatletter \@namedef{subjclassname@2020}{\textup{2020}
	Mathematics Subject Classification} \makeatother

\begin{document}

	\title[A  CP completion theorem]{A minimal completion theorem  and almost everywhere equivalence for completely positive maps}
	
	\author{B. V. Rajarama Bhat}
	\address{Indian Statistical Institute, Stat Math Unit, R V College Post, Bengaluru, 560059, India.}
	\email{bvrajaramabhat@gmail.com, bhat@isibang.ac.in}
	
	\author{Arghya Chongdar}
	\address{Indian Statistical Institute, Stat Math Unit, R V College Post, Bengaluru, 560059, India.}
	\email{chongdararghya@gmail.com}


	
	\maketitle

	\begin{abstract}
		A  problem of completing a linear map on $C^*$-algebras to a
		completely positive map is analyzed. It is shown that whenever such
		a completion is feasible there exists  a unique minimal completion.
		This theorem is used to show that under some very  general
		conditions a completely positive map almost everywhere equivalent to
		a quasi-pure map is actually equal to that map.
		
	\end{abstract} \maketitle
	
	\section{Introduction}

	\maketitle{}

	Motivated by the notion of almost everywhere equality for measurable
	functions Parzygnat and Russo  \cite{PR} came up with a definition
	of almost everywhere equivalence (with respect to a state) for
	completely positive maps. The definition requires the familiar
	concept of   null ideal  and support projection for states and
	completely positive maps and we recall it here.
	
	Let $\mathcal{B}, \mathcal{C}$ be unital $C^*$-algebras and let
	$\xi :\mathcal{B}\to \mathcal {C}$ be a completely positive (CP)
	map. Then
	$$\mathcal{N}_{\xi}:=\{ X \in \mathcal{B}: \xi (X^*X)=0\}$$
	is a left ideal of $\mathcal{B}$. It is called the  {\em null ideal}
	of $\xi .$ If $\mathcal{B}, \mathcal{C}$ are von Neumann algebras
	and $\xi $ is normal then
	\begin{equation}\label{nullspace}\mathcal{N}_{\xi }=\mathcal{B}(I-P):= \{X (I-P) : X\in
		\mathcal{B}\}\end{equation}
	for a unique projection $P$ in
	$\mathcal{B}.$ The projection $P$ is is known as the {\em support
		projection} of $\xi $  (See \cite{Sakai}) and is the smallest
	projection $P$ such that $\xi (P)=\xi (I)$ and satisfies $\xi
	(X)=\xi(XP), \forall X\in \mathcal{B}.$ Now the definition by
	Parzygnat and Russo \cite{PR}  reads as follows.
	
	\begin{defn}\label{def:equal_ae}
		Let $\mathcal{A}, \mathcal{B}, \mathcal{C} $ be  unital
		$C^*$-algebras and let $\xi :\mathcal{B}\to \mathcal{C} $ be a
		unital completely positive map.  Then two linear maps $\varphi ,
		\psi $ from $\mathcal {A}$ to $\mathcal{B}$ are said to be {\em
			equal almost everywhere (a.e.) } with respect to $\xi $ if $\varphi
		(X)-\psi (X)\in \mathcal{N} _{\xi}$ for every $X$ in $\mathcal{A}.$
		This is denoted by $\varphi \underset{\xi}{=} \psi .$
	\end{defn}

	These authors were mainly interested in the case where $\xi $ is a
	state. One of the surprising results of theirs is that on the
	$C^*$-algebra $M_n(\mathbb{C})$ of $n\times n$ complex matrices,
	given any state $\xi $, if a unital completely positive map $\varphi
	$ is equal almost everywhere to the identity map with respect to
	$\xi $, then $\varphi $ must be the identity map (see Theorem 2.48
	of \cite{PR}). The result was obtained through detailed computations
	involving Choi-Kraus coefficients of the given completely positive
	map.  Here we try to have a more conceptual understanding of their
	result by putting the problem in a more abstract setting.

	There is extensive literature on completing partially specified
	matrices to positive matrices (see 
	\cite{Grone},
	\cite{Smith} and Chapter 12 of \cite{Bapat}). It appears in a
	variety of different contexts and has a number of real life
	applications.  Furuta \cite{Furuta} has investigated a completion
	problem of partial matrices whose entries are completely bounded
	maps on a C*-algebra. We have a somewhat different setting. This we
	call as the CP completion problem. The minimal completion theorem in
	Section 3 (See Theorem \ref{Minimal completion theorem}) shows that
	if a linear map admits CP completion then there is a unique minimal
	CP completion. A direct application of this yields a result much
	more general than Theorem 2.48 of \cite{PR}. It is valid for a large
	class of maps called quasi-pure maps between arbitrary
	$C^*$-algebras (See Theorem \ref{main} and Corollary \ref{outcome}).
	Moreover the proof becomes simpler and more algebraic. The notion of
	quasi-pure maps is introduced and studied in Section 2. The
	usefulness of it will be clear in subsequent sections.

	Generalizing the concept of  positive completability from  matrices
	to maps on $C^*$-algebras we have the following definition.
	
	\begin{defn}
		Let $\mathcal{A}, \mathcal{B}$ be  $C^*$-algebras with
		$\mathcal{B}\subseteq \mathscr{B}(\mathcal{H})$ for some Hilbert
		space. Fix $R\in \mathscr{B}(\mathcal{H}).$ Then a linear map $\beta
		:\mathcal{A}\to \mathcal{B}.R:= \{YR: Y\in \mathcal{B}\}$, is said
		to be CP completable (with respect to $R$) if there exists a CP map
		$\varphi :\mathcal{A}\to \mathcal{B}$ such that
		$$\beta(X)= \varphi(X)R, ~~\forall X\in \mathcal{A}.$$
		In such a case, $\varphi $ is called a CP completion of $\beta .$
	\end{defn}
	We are mostly interested in the case when $R$ is a projection. It is
	a natural problem to characterize linear maps which can be CP
	completed.  We give some necessary conditions for CP completability
	in Theorem \ref{completion existence theorem}. More interestingly,
	in the main theorem (Theorem \ref{Minimal completion theorem}) it is
	shown that every CP completable linear map admits a unique `minimal'
	CP completion. The notion of minimality will be made clear below.
	Arguably the setting of this Definition is somewhat peculiar. It is
	motivated by the theory we are able to develop in Section 3 and is
	justified by the applications it has in Section 4.

	We recall a couple of  basic results from the theory of completely
	positive maps. They are well-known and standard.  We state them here
	mainly to introduce the relevant notation. The following theorem is
	due to Stinespring (see Theorem 1, \cite{Stinespring}).
	
	\begin{theorem} (Stinspring's theorem)
		Let $\mathcal{A}$ be a unital $C^*$-algebra with a unit,  let $\mathcal{H}$ be a
		Hilbert space and let $\varphi :\mathcal{A}\to\mathcal{B}(\mathcal{H})$ be a linear map. Then $\varphi $ is completely positive
		if and only if there exists a triple ($\mathcal{K},\pi,V$) where $\mathcal{K}$ is a Hilbert space,
		$\pi:\mathcal{A}\to\mathcal{B}(\mathcal{K})$  is a unital
		$*$-representation  and $V:\mathcal{H}\to\mathcal{K}$ is a bounded
		linear map such that \begin{equation}\label{Stinespring  triple}
			\varphi (X)=V^*\pi(X)V,~\forall X\in \mathcal{A}.   \end{equation}
		Further, $V$ is an isometry if and only if $\varphi $ is unital.
	\end{theorem}
	
	Given a completely positive map $\varphi
	:\mathcal{A}\to\mathcal{B(\mathcal{H})},$ a triple
	($\mathcal{K},\pi,V$) satisfying \ref{Stinespring triple} is called
	a Stinespring triple. Furthermore, a Stinespring triple is called
	minimal if $$\overline{\mbox{span}}\{
	\pi(X)Vh:X\in\mathcal{A},h\in\mathcal{B}(\mathcal{H})\}=\mathcal{K}.$$

	\begin{defn}
		Let $\mathcal{A}, \mathcal{B}$ be  $C^*$-algebras. Then for any two
		CP maps $\varphi,\psi:\mathcal{A}\to\mathcal{B},$ $\psi$ is said to
		be dominated by $\varphi$ (denoted by $\psi\leq\varphi$), if
		$\varphi-\psi$ is CP.
	\end{defn}
	
	Motivated from the Radon Nikodym theorem in measure theory, Arveson proved the following
	theorem for CP maps (Theorem 1.4.2,\cite{Arveson}). Here $\pi
	(\mathcal{A})'$ denotes the commutant of $\pi(\mathcal{A}).$
	
	\begin{theorem} \label{Radon Nikodym type theorem}
		Let $\varphi:\mathcal{A}\to\mathcal{B}(\mathcal{H})$ be a CP map
		with minimal Stinespring triple ($\mathcal{K},\pi,V$). Then a CP map
		$\psi:\mathcal{A}\to\mathcal{B}(\mathcal{H})$ satisfies
		$\psi\leq\varphi$ iff there is a positive contraction $D\in
		~\pi(\mathcal{A})'$ such that $\psi(X)=V^*D\pi(X)V,~\forall X\in
		\mathcal{A}.$ Moreover, given $\psi \leq \varphi $, there exists
		unique such $D$ in $\pi(\mathcal{A})'.$

	\end{theorem}

	Now we are in a good shape to introduce  `minimality' in the context
	of CP completions.
	
	\begin{defn}
		Let $\mathcal{A}, \mathcal{B}$ be  $C^*$-algebras with $\mathcal{B}\subseteq
		\mathscr{B}(\mathcal{H})$ for some Hilbert space ${\mathcal{H}}.$
		Fix $R\in \mathscr{B}(\mathcal{H}).$ Let  $\beta :\mathcal{A}\to
		\mathcal{B}R$ be a linear map. Then a CP completion
		$\varphi:\mathcal{A}\to \mathcal{B}$ is said to be a minimal CP
		completion of $\beta$ if $\psi:\mathcal{A}\to\mathcal{B} $ is any CP
		completion of $\beta $ then  $\varphi \leq \psi $, that is, $\varphi
		$ is dominated by $\psi.$
	\end{defn}

	\section{Quasi-pure CP maps}

	First let us recall the definition of  pure completely positive
	maps.
	\begin{defn}
		A completely positive map $\varphi:\mathcal{A}\to\mathcal{B}$ is
		called pure if $\psi:\mathcal{A}\to\mathcal{B}$ is a CP map with
		$\psi\leq\varphi$ implies $\psi=t\varphi$ for some $t \in[0,1].$
		
	\end{defn}
	
	Before stating a very useful characterization of pure CP maps, we
	recall the notions of cyclic vector of a representation and
	irreducibility of  representations.

	\begin{defn}
		
		Let $\mathcal{A}$ be a $C^*$-algebra and let $\mathcal{K}$ be a
		Hilbert space. Let $\pi:\mathcal{A}\to\mathcal{B}(\mathcal{K})$ be a
		representation of $\mathcal{A}.$ Then a vector $v$ in $\mathcal{K}$
		is called a cyclic vector of the representation $\pi$ if
		$\overline{\mbox{span}}\{\pi(X)v: X\in \mathcal{A}\}=\mathcal{K}.$
	\end{defn}

	\begin{defn}
		
		A representation $\pi:\mathcal{A}\to\mathcal{B}(\mathcal{K})$ on
		$\mathcal{A}$ is called irreducible if every non-zero vector in
		$\mathcal{A}$ is a cyclic vector of $\pi.$
	\end{defn}
	The following result (see Corollary 1.4.3,\cite{Arveson})
	characterizes pure CP maps in terms of irreducibility of their
	Stinespring representation. It can be proved easily using Theorem
	\ref{Radon Nikodym type theorem}.
	
	\begin{prop}
		Let $\varphi:\mathcal{A}\to\mathcal{B}(\mathcal{H})$ be a CP map
		with minimal Stinespring triple ($\mathcal{K},\pi,V$). Then
		$\varphi$ is pure if and only if $\pi$ is an irreducible
		representation.
	\end{prop}
	
	In other words, a CP map $\varphi $ is pure if and only if every
	non-zero vector in the dilation space $\mathcal{K}$ is cyclic for
	the Stinespring representation $\pi .$  Making this condition
	significantly weaker we have the following definition.
	\begin{defn}
		Let $\mathcal{A}$ be a $C^*$-algebra and let $\mathcal{H}$ be a
		Hilbert space.  Let $\varphi :\mathcal{A}\to
		\mathscr{B}(\mathcal{H})$ be a CP map with minimal Stinespring
		triple ($\mathcal{K},\pi,V$). Then $\varphi $ is said to be {\em
			quasi-pure} if every non-zero vector in the range of $V$ is cyclic
		for the representation $\pi $, that is,
		$$\mathcal{K}= \overline{\mbox{span}}\{\pi(X)g: X\in
		\mathcal{A}\},$$
		for every $0\neq g \in V(\mathcal{H}).$
	\end{defn}

	MD Choi in his seminal paper \cite{Choi} showed that any non-zero CP
	map, say, $\varphi : M_{d_1}(\mathbb{C})\to M_{d_2}(\mathbb{C})$ is
	of the form, $\varphi(X) =\sum_{j=1}^{k}L_j^*XL_j, ~~X\in
	M_{d_1}(\mathbb{C})$ such that $k\in \mathbb{N}$ and the set  of
	$d_1\times d_2$ matrices $\{L_1,L_2, \ldots , L_k\}$ forms a
	linearly independent set.  This can be put in Stinespring's
	representation form: $\varphi (X) = V^*\pi(X)V,$ where
	$$\pi (X) = \left[ \begin{array}{cccc}
		X& 0& \ldots & 0 \\
		0&X& \ldots &0\\
		\vdots & \vdots & \ddots & \vdots \\
		0&0& \ldots &X\end{array}\right] _{k\times k},~~V= \left[
	\begin{array}{c}
		L_1\\
		L_2\\
		\vdots \\
		L_k\end{array}\right]_{k\times 1}.$$ In particular, if $\varphi $ is
	pure,  as the representation $\pi $ has to be irreducible, we must
	have $k=1$. Consequently  non-zero pure CP maps have the form
	$$X\mapsto  L^*XL,$$
	for some $d_1\times d_2$ non-zero matrix $L$ and conversely any such
	CP map is pure. Now it is easy to provide examples of  quasi-pure
	maps which are not pure.

	\begin{example}\label{first example}
		Let $d_1, d_2\in \mathbb{N}.$ Let $\rho $ be a positive matrix in
		$M_{d_1}(\mathbb{C})$ and let $v$ be a unit vector in
		$\mathbb{C}^{d_2}.$ Define $\varphi : M_{d_1}(\mathbb{C})\to
		M_{d_2}(\mathbb{C})$  by
		\begin{equation}
			\label{example equation}
			\varphi(X)= \mbox{trace}~(\rho X)|v\rangle \langle v|, ~~X\in
			M_{d_1}(\mathbb{C}),
		\end{equation}
		where $|v\rangle \langle v|$ is the
		projection onto the one dimensional subspace spanned by $v$. Clearly
		$\varphi $ is completely positive.  By spectral theorem, there exist
		orthonormal vectors $u_1, u_2, \ldots , u_k,$ and positive scalars
		$p_1, p_2, \ldots , p_k,$ such that
		$$ \rho = \sum _{j=1}^k p_j|u_j\rangle \langle u_j|,$$
		where $k$ is the rank of $\rho .$ Then $\varphi $ has the minimal
		Stinespring representation, $(\mathcal{K}, \pi , V)$ where
		$\mathcal{K}= \mathbb{C}^{d_1}\otimes \mathbb{C}^k$, $\pi (X)=
		X\otimes I_{\mathbb{C}^k}$ and $V: \mathbb{C}^{d_2}\to \mathcal{K},$
		is defined by $Vh= \sum _{j=1}^k \langle v, h\rangle (\sqrt{p_j}u_j)\otimes e_j,$
		where $e_1, e_2, \ldots , e_k$ is an orthonormal basis of
		$\mathbb{C}^k.$ Now it is not hard to see that $\varphi $ is
		quasi-pure. It is not pure if $k>1.$
	\end{example}

	Here is a characterization of quasi-pure CP maps without referring
	to their Stinespring dilations and just using kernels of some maps.
	
	\begin{theorem}\label{kernel}
		Let $\mathcal{A}$ be a $C^*$-algebra and let $\mathcal{H}$ be a
		Hilbert space.  Let $\varphi :\mathcal{A}\to
		\mathscr{B}(\mathcal{H})$ be a CP map. Then $\varphi $ is quasi-pure
		if and only if whenever a non-zero CP map $\alpha$ is dominated by
		$\varphi$,
		$\mbox{ker}(\alpha (1))= \mbox{ker} (\varphi (1)).$
	\end{theorem}
	\begin{proof}
		
		Let $(\mathcal{K}, \pi , V)$ be a  minimal Stinespring
		triple of $\varphi .$
		
		First assume that $\varphi :\mathcal{A}\to \mathscr{B}(\mathcal{H})$ is a quasi-pure CP map. Let $\alpha :\mathcal{A}\to
		\mathscr{B}(\mathcal{H})$ be a non-zero CP map such that
		$\alpha\leq\varphi.$   As $\alpha (1)\leq \varphi (1)$,
		$ker(\varphi (1))\subseteq ker (\alpha (1)).$ By Radon-Nikodym type
		theorem (Theorem \ref{Radon Nikodym type theorem}), there exists a
		positive contraction $D\in\pi(\mathcal{A})'$    such that
		$$\alpha(X) = V^*D\pi(X)V,~\text{for all}~ X\in\mathcal{A}.$$
		If we assume that  $ker(\alpha (1))\subsetneqq ker (\varphi (1)),$
		there exists a non-zero $h_0\in \mathcal{H}$ such that $DVh_0=0$ but
		$Vh_0\neq0.$ This implies, $D\pi(X)Vh_0=0,\forall~X\in\mathcal{A}.$
		Recalling that ($\mathcal{K},\pi,V$) is a minimal Stinespring triple
		for the  quasi-pure CP map $\varphi,$ $\mathcal{K}=\overline{\mbox{span}}\{\pi
		(X)Vh:X\in
		\mathcal{A}, h\in \mathcal{H}\}=\overline{\mbox{span}}\{\pi (X)Vh_0:X\in  \mathcal{A}\}.$
		Thus $D\pi(X)V=0,~\forall~X\in\mathcal{A}.$ Consequently $\alpha=0.$
		This contradicts the hypothesis that $\alpha$ is non-zero. Therefore, $ker(\alpha (1))= ker (\varphi (1)).$
		
		Now to prove the other implication,  assume that  the CP map
		$\varphi:\mathcal{A}\to \mathscr{B}(\mathcal{H})$ is not quasi-pure,
		then we have  $h_0\in\mathcal{H}$ such that $0\neq Vh_0$ is not
		cyclic for $\pi .$ Let
		$\tilde{\mathcal{K}}=\overline{\mbox{span}}\{\pi(X)Vh_0:X\in
		\mathcal{A}\}$ and let $P$ be the projection of $\mathcal{K}$ onto
		this reducing subspace $\tilde{\mathcal{K}}.$ Set
		$$\alpha (X)= V^*\pi (X)(I-P) V, ~X\in \mathcal{A}.$$
		Then $\alpha:\mathcal{A}\to \mathscr{B}(\mathcal{H})$ defines a CP
		map dominated by $\varphi.$ Clearly, $\alpha (1)h_0=0$ as $Vh_0\in
		\tilde{\mathcal{K}}.$ But $\varphi(1)h_0=V^*Vh_0\neq 0$ as $\langle
		h_0, V^*Vh_0\rangle = \|Vh_0\|^2\neq 0.$  Hence $\mbox{ker}(\alpha
		(1))\neq \mbox{ker}(\varphi (1)).$
		
	\end{proof}
	As an immediate consequence of this theorem we have the following
	corollary.
	
	\begin{corollary}\label{order}
		Suppose $\varphi : \mathcal{A}\to \mathscr{B}(\mathcal{H})$ is a
		quasi-pure CP map and $\alpha :\mathcal{A}\to
		\mathscr{B}(\mathcal{H})$ is a  CP map dominated by $\varphi .$ Then
		$\alpha $ is quasi-pure.
	\end{corollary}
	\begin{proof}
		Clear from Theorem \ref{kernel}.
	\end{proof}
	
	For CP maps on matrix algebras quasi-purity can be described in
	terms of Choi-Kraus coefficients.

	\begin{theorem}\label{matrix alg}
		
		Let $\varphi:M_{d_1}(\mathbb{C})\to M_{d_2}(\mathbb{C})$ be a CP map
		with a minimal Choi-Kraus decomposition: \begin{equation}
			\label{Choi-Kraus} \varphi (X) =\sum_{j=1}^{k}L_j^*XL_j, ~~X\in
			M_{d_1}(\mathbb{C})\end{equation} where $k\in \mathbb{N}$ and $L_1,
		L_2, \ldots , L_k$ are $d_1\times d_2$ matrices.  Then $\varphi$ is
		\em{quasi-pure} iff the collection $\{L_1h_0, L_2h_0, \ldots ,
		L_kh_0\}$ is linearly independent, whenever $L_ih_0\neq 0$ for some
		$i$ in $\{1, 2, \ldots , k\}$ and $h_0\in \mathbb{C}^{d_2}.$
		
	\end{theorem}
	
	\begin{proof} Here $\varphi$ has the
		minimal Stinespring representation, $(\mathcal{K}, \pi , V)$ where
		$\mathcal{K}= \mathbb{C}^{d_1}\otimes \mathbb{C}^k$, $\pi (X)=
		X\otimes I_{\mathbb{C}^k}$ and $V: \mathbb{C}^{d_2}\to \mathcal{K},$
		is defined by $Vh= \sum _{j=1}^k L_jh\otimes e_j,$ where $e_1, e_2,
		\ldots , e_k$ is an orthonormal basis of $\mathbb{C}^k.$ Assume that
		there exists a nonzero $h_0\in \mathbb{C}^{d_2}$ such that the
		collection
		$\{L_1h_0, L_2h_0, \ldots , L_kh_0\}$ is linearly dependent, and $L_ih_0\neq 0$ for some $i$ in
		$\{1, 2, \ldots , k\}.$ Then $Vh_0$ is non-zero and
		there exist complex numbers $a_j$'s (not all  zero) such that
		\begin{equation}\label{linear}
			\sum_{j=1}^{k}a_jL_jh_0=0.\end{equation} Let
		$\tilde{\mathcal{K}}=\overline{\mbox{span}}\{\pi (X)Vh_0:X\in
		M_{d_1}(\mathbb{C}) \} =\overline{\mbox{span}}\{\sum _{j=1}^k
		X(L_jh_0)\otimes e_j:X\in M_{d_1}(\mathbb{C}) \}$, a subspace of the
		minimal dilation space, $\mathbb{C}^{d_1}\otimes\mathbb{C}^{k}.$ Now
		the condition (\ref{linear}) implies that every vector in
		$\tilde{\mathcal{K}}$ is orthogonal to vectors of the form $\sum
		_{j=1}^k\bar{a}_jv\otimes e_j$ for any $v\in \mathbb{C}^{d_1}.$ In
		particular, $\tilde{\mathcal{K}}$ is not whole of the dilation space
		and $\varphi $ is not quasi-pure.
		
		Now to prove the other implication, suppose  $h_0$ is an element of
		$\mathbb{C}^{d_2}$  such that the collection $\{L_1h_0, L_2h_0,
		\ldots , L_kh_0\}$ is a linearly independent subset of
		$\mathbb{C}^{d_1}$. Extend it to a basis $\{ f_1, f_2, \ldots ,
		f_{d_1}\}$ of $\mathbb {C}^{d_1},$ where $f_i=L_ih_0, i=1,2,
		\ldots,k.$ Fix a basis $\{ e_1, e_2, \ldots , e_k\}$ for
		$\mathbb{C}^k.$  Define linear maps $X_{ir}\in M_{d_1}(\mathbb{C})$
		by
		$$X_{ir}(f_j)=\delta_{r,j}f_{i}:i,j,r\in \{1,2, \ldots,d_1\},$$ where
		$\delta $ is the Kronecker delta function.
		Now  \begin{eqnarray*}
			\mathbb{C}^{d_1}\otimes \mathbb{C}^k & \supseteq  &
			\overline{\mbox{span}}\{\pi (X)Vh_0:X\in M_{d_1}(\mathbb{C})
			\}\\
			&=& \overline{\mbox{span}}\{\sum _{j=1}^k X(L_jh_0)\otimes e_j:X\in
			M_{d_1}(\mathbb{C}) \} \\
			&=& \{ \sum _{j=1}^kXf_j\otimes e_j: X\in M_{d_1}(\mathbb{C})\}\\
			&\supseteq & \{ \sum _{j=1}^kX_{ir}f_j\otimes e_j: 1\leq i,r\leq
			d_1\}
			\\& = & \{ f_i\otimes e_j: 1\leq i\leq
			d_1,1\leq j\leq k\}\\
			&=& \mathbb{C}^{d_1}\otimes \mathbb{C}^k.
		\end{eqnarray*}
		Therefore, $Vh_0$ is a cyclic vector for the representation
		$\pi$, as required. \end{proof}
	
	Recall that the Choi rank of a CP map $\varphi $ on matrix algebras
	is the minimum number of $L_j$'s needed in expressing it in the form
	$\varphi(X) = \sum _{j}L_j^*XL_j.$ The Choi rank of a CP map
	$\varphi :M_{d_1}(\mathbb{C})\to M_{d_2}(\mathbb{C})$ is at most
	$d_1d_2.$
	
	\begin{corollary}
		Let $\varphi:M_{d_1}(\mathbb{C})\to M_{d_2}(\mathbb{C})$ be a
		quasi-pure CP map. Then the Choi rank of $\varphi $ is at most
		$d_1$.
	\end{corollary}
	\begin{proof} This follows from Theorem \ref{matrix alg} due to
		linear independence of $\{L_1h_0, L_2h_0, \ldots , L_kh_0\}$ in
		$\mathbb{C}^{d_1}.$
	\end{proof}
	
	\begin{corollary}
		Let $\varphi:M_{d_1}(\mathbb{C})\to M_{d_2}(\mathbb{C})$ be a
		quasi-pure CP map with Choi-Kraus decomposition as in
		(\ref{Choi-Kraus}). Then $\mbox{ker}~(L_i)=~\mbox{ker}~(L_j)$ for
		$1\leq i,j\leq k.$
	\end{corollary}
	\begin{proof}
		From Theorem \ref{matrix alg}, if $L_ih_0\neq 0$ for some $i$,  then
		$L_jh_0\neq 0$ for all $j$.
	\end{proof}
	
	A minimal Choi-Kraus decomposition of the CP map of  Example
	$\ref{first example}$ is given by,
	$$\varphi (X) = \sum _{j=1}^kL_j^*XL_j, ~~X\in
	M_{d_1}(\mathbb{C}),$$ where $L_j= |\sqrt{p_j}u_j\rangle \langle
	v|$, is the map $w\mapsto \sqrt{p_j} u_j\langle v, w\rangle.$ This
	shows that the Choi rank of a quasi-pure map from
	$M_{d_1}(\mathbb{C})$ to $M_{d_2}(\mathbb{C})$ can be any number in
	$\{1,2, \ldots , d_1\}.$
	
	A CP map $\varphi :M_{d_1}(\mathbb{C})\to M_{d_2}(\mathbb{C})$ is  defined
	to be {\em entanglement breaking} (EB)
	if it has a (not-necessarily minimal) Choi-Kraus decomposition of
	the form $\varphi (X)= \sum _{j=1}^kL_j^*XL_j$, where $L_j$'s are
	rank one operators. Entanglement breaking maps have a special role
	in quantum information theory and there is plenty of literature on
	the same. See  the influential paper of Horodecki, Shor and Ruskai
	\cite {HSR}, and its references as well as  citations for further
	information. It is clear from this definition that the quasi-pure CP
	map considered in Example \ref{first example} is an EB map.
	Conversely, from Theorem \ref{kernel}   or from Theorem \ref{matrix
		alg} it follows that every quasi-pure EB map $\varphi
	:M_{d_1}(\mathbb{C})\to M_{d_2}(\mathbb{C})$ is necessarily of the
	form \ref{example equation}  for some positive matrix $\rho\in
	M_{d_1}(\mathbb{C})$ and unit vector $v\in \mathbb{C}^{d_2}$.

	\section{Minimal completion theorem}
	
	Consider a $2\times 2$ block
	operator matrix of operators on Hilbert spaces,
	$$ N= \left[\begin{array}{cc}
		A&B\\
		C&D\end{array}\right].$$ Assume  that $A$ is strictly positive, that
	is, it is positive and invertible.  Then $N$ is positive  iff $C=B^*
	$ and $D\geq CA^{-1}B$ (See Theorem 1.3.3 of \cite{Bhatia}. This
	book deals with only matrices. But this theorem can  easily be
	generalized to operators on Hilbert spaces). Now consider a
	partially given block operator matrix:
	$$M=\left[\begin{array}{cc}
		A&*\\
		C&*\end{array}\right],$$  where $*$ indicates that the particular
	entry hasn't been specified. Now it is clear that if $A$ is strictly
	positive, then $M$ can always be completed to a positive block
	operator matrix and in fact there is a minimal positive completion
	given by
	$$\left[\begin{array}{cc}
		A&C^*\\
		C&CA^{-1}C^*\end{array}\right].$$ Little more care is needed when
	$A$ is positive but not strictly positive.
	
	\begin{lemma}\label{completion existence lemma}
		Let $\mathcal{K}$ be a proper closed subspace of a Hilbert space
		$\mathcal{H}.$ Suppose $A\in \mathscr{B}(\mathcal{K}), C\in
		\mathscr{B}(\mathcal{K}, \mathcal{K}^\perp)$ are given.Then there
		exists a positive operator, $M\in \mathscr{B}(\mathcal{H})$ such
		that with respect to the decomposition
		$\mathcal{H}=\mathcal{K}\oplus \mathcal{K}^\perp $, it has the form:
		$$M= \left[\begin{array}{cc}
			A&*\\
			C&*\end{array}\right],$$ where $*$ denotes unspecified entries, if
		and only if there exists $q>0$ such that $$C^*C\leq qA.$$ Moreover,
		in such a case, there exists unique minimal $D$ such that
		$\left[\begin{array}{cc}
			A&C^*\\
			C&D\end{array}\right]$ is positive.
	\end{lemma}

	\begin{proof}
		Suppose
		$$M= \left[\begin{array}{cc}
			A&B\\
			C&D\end{array}\right],$$ is a positive operator. Then by Proposition
		1.3.2 of  \cite{Bhatia},  $A, D$ are positive, and there exists a
		contraction $K\in \mathscr{B}(\mathcal{K}, \mathcal{K}^\perp)$ such
		that $ C = D^{\frac{1}{2}}KA^{\frac{1}{2}}.$ Therefore,
		$$C^*C= A^{\frac{1}{2}}K^*DKA^{\frac{1}{2}}\leq qA$$
		where $q= \|K^*DK\|.$ Conversely, suppose $C^*C\leq qA$ for some
		$q>0$. Let $C=V|C|$ be the polar decomposition of $C$. Then
		$$
		\left[\begin{array}{cc}
			A&C^*\\
			C&qVV^*\end{array}\right]= \left[\begin{array}{cc}
			\frac{1}{\sqrt{q}}|C|&0\\
			0&\sqrt{q}I\end{array}\right].
		\left[\begin{array}{cc}
			I&V^*\\
			V&VV^*\end{array}\right].\left[\begin{array}{cc}
			\frac{1}{\sqrt{q}}|C|&0\\
			0&\sqrt{q}I\end{array}\right]+ \left[\begin{array}{cc}
			A-\frac{1}{q}C^*C&0\\
			0&0\end{array}\right]\geq 0.$$ As $V$ is a partial isometry
		$qVV^*\leq qI.$ By functional calculus for $0<s<t<\infty $, $(A+s),
		(A+t)$ are invertible and $(A+s)^{-1}\geq (A+t)^{-1}.$ Also,
		$$\left[\begin{array}{cc}
			A+t&C^*\\
			C&qI\end{array}\right]\geq 0,$$ implies that $C(A+t)^{-1}C^*\leq
		qI.$ Therefore, as $t\downarrow 0$, $C(A+t)^{-1}C^*$ is a
		monotonically increasing family bounded by $qI$. Consequently,
		$C(A+t)^{-1}C^*\uparrow D$ in SOT for some positive bounded operator
		$D$. Observe that, as $C(A+t)^{-1}C^*\leq D$ for every $t>0$,
		$$\left[\begin{array}{cc}
			A+t&C^*\\
			C&D\end{array}\right]\geq 0.$$ Taking strong operator topology limit
		as $t$ decreases to $0,$
		$$\left[\begin{array}{cc}
			A&C^*\\
			C&D\end{array}\right]\geq 0.$$ Similar arguments show that this is
		the unique minimal solution. \end{proof}
	
	These observations lead to some necessary conditions for the existence of
	CP completion of linear maps. Here and elsewhere, by $\mathcal{A}_+$
	we mean positive elements of $\mathcal{A}.$
	
	\begin{theorem} \label{completion existence theorem}
		Let $\mathcal{A},
		\mathcal{B}$ be $C^*$-algebras with $\mathcal{B}\subseteq
		\mathscr{B}(\mathcal{H})$ for some Hilbert space ${\mathcal{H}}.$
		Fix a projection $R\in \mathcal{B}$. Let $\beta :\mathcal{A}\to
		\mathcal{B}.R$ be a linear map.  If $\beta $ is CP completable then
		the following properties hold:
		
		(i) The map $X\mapsto R\beta(X)$ is completely positive;
		
		(ii) There exists $q\geq 0$ such that,
		$$\beta(X)^*(1-R)\beta(X)\leq q\|X\|  R\beta(X), ~~\forall X \in \mathcal{A}_+.$$
		
		(iii) For $X=[X_{ij}]_{1\leq i,j\leq n}$ in $M_n(\mathcal{A})_+$,
		$[\beta(X_{ij})]_{1\leq i,j\leq n}$ in $M_n(\mathcal{B}R)$ is
		positive completable to an operator in $M_n(\mathcal{B}).$
		
	\end{theorem}

	\begin{proof}
		Here $\beta :\mathcal{A}\to \mathcal{B}.R$ is given to be CP
		completable. Let, $\varphi:\mathcal{A}\to\mathcal{B}(\mathcal{H})$
		be one such completion for $\beta,$ i.e
		$\beta(X)=\varphi(X)R,~~\forall X\in\mathcal{A}.$
		
		(i) Then the map $X\mapsto R\beta(X)=R\varphi(X)R$ is completely
		positive.

		(ii) For every $X \in \mathcal{A} _+,$  with respect to the
		decomposition $\mathcal{H}= R(\mathcal{H})\oplus
		(I-R)(\mathcal{H})$,
		$$\varphi(X)=
		\left[\begin{array}{cc}   R\varphi(X)R&R\varphi(X)(1-R)\\
			(1-R)\varphi(X)R&(1-R)\varphi(X)(1-R)\end{array}\right],$$ is
		positive. Now as in the proof of  Lemma \ref{completion existence
			lemma},
		$$R\varphi(X)^*(1-R)\varphi(X)R\leq \|(1-R)\varphi(X)(1-R)\| R\varphi(X)R \leq q \|X \| R\varphi(X)R, $$ where $q= \|\varphi \|.$
		Hence we have,
		$$\beta(X)^*(1-R)\beta(X)\leq q \|X\| R\beta(X), ~~\forall X \in
		\mathcal{A}_+.$$

		(iii) For $X=[X_{ij}]_{1\leq i,j\leq n}\geq 0$ in
		$M_n(\mathcal{A}),$ $[\varphi(X_{ij})]_{1\leq i,j\leq n}$ in
		$M_n(\mathcal{B})$ is positive and $[\beta(X_{ij})]_{1\leq i,j\leq
			n}=[\varphi(X_{ij})R]_{1\leq i,j\leq n}=[\varphi(X_{ij})]_{1\leq
			i,j\leq n}(1\otimes R).$
		Therefore, $[\beta(X_{ij})]_{1\leq i,j\leq n}$ is positive completable. \end{proof}

	In the following result we observe that the condition (iii) of this
	theorem is also sufficient in some special cases.

	\begin{theorem} \label{completion existence theorem}
		Let $\mathcal{A}=M_d(\mathbb{C})$ for some $d\geq 1.$ Let
		$\mathcal{B}$ be a $C^*$-algebra with $\mathcal{B}\subseteq
		\mathscr{B}(\mathcal{H})$ for some Hilbert space ${\mathcal{H}}.$
		Fix a projection $R\in \mathcal{B}$. Let $\beta :\mathcal{A}\to
		\mathcal{B}.R$ be a linear map.  Then $\beta $ is CP completable if
		and only if $[\beta(E_{ij} )]_{1\leq i,j\leq d}$ can be completed to
		a positive matrix of $M_d(\mathcal{B})$ where $E_{i,j}, 1\leq
		i,j\leq d$ are the matrix units of $M_d(\mathbb{C}).$ \end{theorem}
	\begin{proof}
		If $[Y_{ij}]$ is  a positive completion of $[\beta(E_{ij})]$, define
		a linear map $\varphi :M_d\to \mathcal{B}$ by setting $\varphi
		(E_{ij})=Y_{ij}, 1\leq i,j\leq d.$ Then by a well-known result of
		Choi (See \cite{Choi} ), $\varphi  $ defines a completely positive
		map. By linearity it is clear that $\varphi  $ is a completion of
		$\beta .$ The converse has been observed in the previous Theorem.
	\end{proof} So far we do not have a general necessary and sufficient
	condition for CP completability. However, the following result shows
	that if a linear map admits a CP completion then there is a unique
	minimal CP completion. This is the main result of this Section.

	\begin{theorem} \label{Minimal completion theorem}
		\textbf{(Minimal Completion theorem):} Let $\mathcal{A},
		\mathcal{B}$ be  $C^*$-algebras with $\mathcal{B}\subseteq
		\mathscr{B}(\mathcal{H})$ for some Hilbert space ${\mathcal{H}}.$
		Fix $R\in \mathscr{B}(\mathcal{H}).$ Let  $\beta :\mathcal{A}\to
		\mathcal{B}.R$
		be CP completable. Then there  exists a unique CP completion  $\alpha
		$ such that if $\psi $ is any CP completion of $\beta $ then
		$\alpha $ is dominated by $\psi .$
	\end{theorem}
	\begin{proof}
		Let $\varphi _i : \mathcal{A}\to \mathcal{B}$ be  CP completions of
		$\beta $ for $i=1,2.$  Let $(\mathcal{K}_i, \pi _i, V_i)$ be the
		unique minimal Stinespring representations of $\varphi _i,$ so that
		\begin{eqnarray*}
			\varphi _i(X) &=& V_i^*\pi _i(X)V_i, ~~X\in \mathcal{A};\\
			\mathcal{K}_i&=& \overline{\mbox{span}}\{\pi _i(X)V_ih:X\in
			\mathcal{A}, h\in \mathcal{H}\},
		\end{eqnarray*} for $i=1,2.$ Now take $ \tilde{\mathcal{K}_i}=\overline{\mbox{span}}\{\pi
		_i(X)V_iRh:X\in \mathcal{A}, h\in \mathcal{H}\}$. Note that
		$\tilde{\mathcal{K}_i}$ is a reducing subspace for the
		representation $\pi _i.$ Let $Q_i$ be the projection of
		$\mathcal{K}_i$ to $\tilde{\mathcal{K}_i}.$ Set $$\alpha _i(X)=
		V_i^*\pi _i(X)Q _iV _i, ~X\in \mathcal{A}.$$ Then  $\alpha _i$ is a
		CP map dominated by $\varphi _i$. Also $\alpha _i(X)R=\varphi _i
		(X)R= \beta (X).$ Hence it is a completion of $\beta .$  We need to
		show that $\alpha _i$ is independent of $i.$ Define
		$U:\tilde{\mathcal {K}_1}\to \tilde{\mathcal{K}_2}$ by setting
		$$U\pi _1(X)V_1Rh= \pi _2(X)V_2Rh, X\in \mathcal{A}, h\in \mathcal{H}.$$
		By direct computation, $U$ is isometric as:
		\begin{eqnarray*}
			\langle \pi _i(X)V_iRg, \pi _i(Y)V_iRh\rangle &=& \langle Rg,
			V_i^*\pi _i(X^*Y)V_iRh\rangle \\
			&=& \langle Rg, \varphi _i(X^*Y)Rh\rangle\\
			&=& \langle Rg,
			\beta(X^*Y)Rh\rangle ,
		\end{eqnarray*}
		for every $X,Y\in \mathcal{A}, g,h\in \mathcal{H}.$ From the
		definition of $\tilde{\mathcal{K}_1}, \tilde{\mathcal{K}_2}$ and $U$, $U$
		extends to a unitary and satisfies the following identity : $$U\pi_1(X)Q_1=Q_2\pi_2(X)U,~\forall~X\in\mathcal{A}.$$
		Using the definition of $U$,  and the fact
		that $\beta(X)h=\varphi _1(X)R=\varphi _2(X)R$, we have,
		$$\beta(X)h=V_2^*Q_2\pi _2(X)V_2Rh=V_2^*Q_2U\pi _1(X)V_1Rh= V_1^*Q_1\pi
		_1(X)V_1Rh, \forall X\in \mathcal{A}, h\in \mathcal{H}.$$ Since the
		collection of vectors of the form $\pi _1(X)V_1Rh$ is total in
		$\tilde{K}_1$, we get $V_2^*Q_2U=V_1^*Q_1$ or equivalently
		$U^*Q_2V_2=Q_1V_1.$ Hence $ Q_2V_2=UQ_1V_1.$ Now, for all
		$X\in\mathcal{A},$
		$$\alpha _2(X)= V_2^*Q_2\pi _2(X)Q_2V_2=
		V_2^*Q_2\pi _2(X)UQ_1V_1=V_2^*Q_2U\pi _1(X)Q_1V_1=V_1^*Q_1\pi
		_1(X)Q_1V_1=\alpha _1 (X).$$ \end{proof}

	Instead of assuming that $\mathcal{B}$ is a subalgebra of
	$\mathcal{B}(\mathcal{H})$ for some Hilbert space $\mathcal{H}$ and
	using the Stinespring representation, we could have used the Hilbert
	$C^*$-module language  and Paschke's version of Stinespring's
	theorem (See \cite{Paschke} ) in this proof. We have not opted for
	such an approach to keep the presentation more accessible.

	\section{Almost everywhere equivalence for CP maps}
	
	Generalizing the notion of equal almost everywhere with respect to
	a state for CP maps, we have the following definition.
	
	\begin{defn} Let $\mathcal{A}$ be a $C^*$-algebra and let
		$\mathcal{H}$ be a Hilbert space. Let $\varphi , \psi$ be linear
		maps from $\mathcal{A}$ to $\mathscr{B}(\mathcal{H}).$ Suppose $R\in
		\mathscr{B}(\mathcal{H}).$ Then $\psi $ is said to be $R$-equivalent
		to $\varphi $ if
		\begin{equation}\label{R-equivalence}
			\varphi(X)R= \psi(X)R, ~~\forall X\in \mathcal{A}.
		\end{equation}
		This is denoted by $\varphi \underset{R}{=} \psi .$
	\end{defn}
	
	Comparing with Definition \ref{def:equal_ae},  $\varphi $ is equal
	almost everywhere to $\psi $ with respect to a CP map $\xi $, if and
	only if $\varphi \underset{P}=\psi $ where $P$ is the support
	projection of $\xi .$ Note that $\varphi \underset{R}=0$ if and only
	if $\varphi(X)R=0$ for every $X\in \mathcal{A}.$

	\begin{theorem}\label{decomposition}
		Let $\mathcal{A}$ be a unital $C^*$-algebra, $\mathcal{H}$ be a
		Hilbert space and let $\varphi :\mathcal{A}\to
		\mathscr{B}(\mathcal{H})$
		be a completely positive map. Then for any
		$R~\in\mathscr{B}(\mathcal{H}),$ there exists a unique CP map
		$\alpha:= \alpha _R$ with following two decomposition properties:
		\begin{enumerate}
			\item $\varphi$ decomposes as $\varphi =\alpha
			+\varphi _1$ where $\alpha , \varphi _1$ are CP maps satisfying
			\begin{equation}
				\varphi \underset{R}{=} \alpha  ,~~\varphi _1\underset{R}{=} 0.
			\end{equation}
			\item If $\psi :\mathcal{A}\to \mathscr{B}(\mathcal {H})$ is any CP map
			such that $ \varphi \underset{R}{=} \psi $, then $\psi $ decomposes
			as $\psi =\alpha +\psi _1$ with these CP maps satisfying
			\begin{equation}
				\varphi \underset{R}{=} \alpha \underset{R}{=} \psi , ~~ \varphi
				_1\underset{R}{=} \psi _1\underset{R}{=} 0.
			\end{equation}
		\end{enumerate}
	\end{theorem}
	\begin{proof}
		Define $\beta :\mathcal{A}\to \mathscr{B}(\mathcal{H})$ by $\beta
		(X)= \varphi(X)R.$ Now $\beta $ is CP completable with $\varphi $
		being one such completion. Hence Theorem~\ref{Minimal completion theorem} is
		applicable. Let $(\mathcal{K}, \pi , V)$ be a minimal Stinespring
		representation of $\varphi .$ Take $
		\tilde{\mathcal{K}}=\overline{\mbox{span}}\{\pi (X)VRh:X\in
		\mathcal{A}, h\in \mathcal{H}\}$. Let $Q$ be the projection of
		$\mathcal{K}$ onto the reducing subspace $\tilde{\mathcal{K}}.$ Set
		$$\alpha (X)= V^*\pi (X)Q V, ~X\in \mathcal{A}.$$ Then
		$\alpha $ is a CP map dominated by $\varphi $. Now
		Theorem~\ref{Minimal completion theorem} implies that $\alpha$ is
		the minimal CP completion of $\beta $.
		Take $\varphi _1=\varphi -\alpha $ and $\psi _1= \psi -\alpha $.
		Now decomposition properties are easy verifications. The uniqueness
		follows from the uniqueness of minimal completion.
	\end{proof}

	This leads to the main theorem of this Section. It shows some kind of rigidity of quasi-pure maps, in the
	sense  that under
	some mild conditions any CP map which is $R$-equivalent to a
	quasi-pure CP map is the map itself.
	\begin{theorem}\label{main}
		Let $\mathcal{A}$  be a unital $C^*$-algebra and let $\mathcal{H}$
		be a Hilbert space.  Let $\varphi :\mathcal{A}\to
		\mathscr{B}(\mathcal{H}), \psi :\mathcal{A}\to
		\mathscr{B}(\mathcal{H})$ be completely positive maps, where
		$\varphi $ is quasi-pure and $\varphi(I)=\psi(I)$. Suppose some
		$R\in \mathscr{B}(\mathcal{H})$ is given such that $\varphi $ is
		$R-$equivalent to $\psi$, that is,
		\begin{equation}\label{R} \varphi(X)R=\psi(X)R, ~~\forall X\in
			\mathcal{A}.
		\end{equation}
		Assume that there exists $X_0\in \mathcal{A}$ such that
		$\varphi(X_0)R\neq 0.$  Then $\varphi =\psi .$
	\end{theorem}
	
	\begin{proof}
		We view $\varphi , \psi $ as CP completions of the map $\beta $
		defined on $\mathcal{A}$  by $\beta (X)= \varphi(X)R=\psi(X)R.$ Let
		$(\mathcal{K}, \pi , V)$ be a minimal Stinespring representation of
		$\varphi $.  The existence of $X_0$ as above implies that $\pi
		(X_0)VR\neq 0$, Hence the reducing subspace $\{\pi (X)VRh: X\in
		\mathcal{A}, h\in \mathcal{H}\}$ is non-trivial and the quasi-purity
		of $\varphi $ implies that it is whole of $\mathcal{K}.$ Therefore
		the minimal CP completion of the map $\beta $ is same as $\varphi .$
		In particular, $\varphi $ is dominated by $\psi .$ Since $\varphi
		(I)=\psi (I)$, it follows that $\varphi =\psi .$
		
	\end{proof}
	
	In this theorem the condition of quasi-purity plays a very crucial
	role as the following example shows.
	
	\begin{example}\label{non quasi-pure}
		Let $\mathcal{A}$  be a unital $C^*$-algebra and let $\mathcal{H}$
		be a Hilbert space.  Let $\varphi :\mathcal{A}\to
		\mathscr{B}(\mathcal{H})$ be a  completely positive map which is not
		quasi-pure. Then by  Theorem \ref{kernel} there exists a non-zero CP
		map $\alpha $ dominated by $\varphi $, with an $h_0\in \mathcal{H}$
		such that $\alpha (1)h_0=0$ and $\varphi (1)h_0\neq 0.$
		
		Suppose we can choose $Z\in \mathscr{B}(\mathcal{H})$ such that
		$Zh_0=0$ and $Z^*\alpha (X_1)Z\neq \alpha (X_1)$ for some $X_1\in
		\mathcal{A}$, then $\psi :\mathcal{A}\to \mathscr{B}(\mathcal{H})$
		defined by
		$$\psi (X) = Z^*\alpha (X)Z+ (\varphi -\alpha )(X), ~~X\in
		\mathcal{A},$$ satisfies $\varphi (X)R=\psi (X)R$ where
		$R=|h_0\rangle \langle h_0|$, but $\psi $ is different from $\varphi
		$. Typically, it is possible to choose such a $Z$, even with the
		additional restriction $Z^*\alpha (1)Z= \alpha (1)$ so that $\psi
		(1)=\varphi (1)$ and  all the conditions of the previous theorem are
		satisfied except quasi-purity of $\varphi $. However this may not be
		possible in some special cases as the following example
		demonstrates.
	\end{example}
	
	\begin{example}\label{special}
		Consider $\varphi :M_2(\mathbb{C})\to M_2(\mathbb{C})$ defined by
		$$\varphi (
		\left[ \begin{array}{cc} a&b\\
			c&d\end{array}\right]) = \left[
		\begin{array}{cc}
			a+d& b+c\\ b+c& a+d\end{array}\right].$$ Then $\varphi $ is CP. It
		is not hard to see that $\varphi $ is not quasi-pure.Take, $$R=\left[ \begin{array}{cc} 1&0\\
			0&0\end{array}\right] .$$ It is easily seen that if $\psi $ is a CP
		map satisfying $\varphi (X)R=\psi (X)R$ for all $X$ and $\psi
		(1)=\varphi(1)$, then $\psi =\varphi .$ In other words, the
		conclusion of the previous theorem may hold for some $R$ even though
		$\varphi $ is not quasi-pure.
	\end{example}

	\begin{corollary}\label{outcome}
		Let $\mathcal{A}$ be a unital $C^*$-algebra and let $\mathcal{B},
		\mathcal{C}$ be von Neumann algebras. Let $\xi : \mathcal{B}\to
		\mathcal{C}$ be a non-zero normal completely positive map. Let
		$\varphi, \psi : \mathcal{A}\to \mathcal{B}$ be two completely
		positive maps. Suppose $\varphi$  is a quasi-pure completely
		positive map, $\xi\circ \varphi $ is a non-zero map, $\varphi(I)
		=\psi(I)$ and $\varphi \underset{\xi}{=} \psi$. Then $\varphi
		=\psi$.
	\end{corollary}

	\begin{proof} As $\mathcal{B}$ is a von Neumann algebra,
		$\mathcal{B}\subseteq \mathscr{B}(\mathcal{H})$ for some Hilbert
		space $\mathcal{H}$. In order to apply the previous results we may
		view $\varphi , \psi $ as maps from $\mathcal{A}$ to
		$\mathscr{B}(\mathcal{H}).$  Let $P$ be the support projection of
		$\xi .$ Then by (\ref{nullspace})
		$\mathcal{N}_{\xi}=\mathcal{B}(I-P).$ From
		$\varphi\underset{\xi}{=}\psi,$ we get $\varphi(X)-\psi(X)\in
		\mathcal{B}(I-P)$ for all $X\in \mathcal{A}.$ In particular,
		$(\varphi(X)-\varphi (X))P=0$, or $\varphi(X)P=\psi(X)P$ for all
		$X\in \mathcal{A}.$  Now the result is immediate from the previous
		theorem.
	\end{proof}
	
	This corollary generalizes results of \cite{PR} in the following
	ways. The algebras are no longer just matrix algebras. The identity
	map has been replaced by quasi-pure maps and the states are replaced
	by general completely positive maps.

	The following example shows that in the last result the condition of
	normality on the CP map $\xi $  is not redundant.
	
	\begin{example} Let $\mathcal{H}$ be a separable infinite dimensional Hilbert space. Let
		$\mathscr{K}(\mathcal{H}) $ be  the algebra of compact operators and
		let
		$\pi:\mathscr{B}(\mathcal{H})\to\mathscr{B}(\mathcal{H})/\mathscr{K}(\mathcal{H})
		$ be the quotient map. Fix a state $ \xi_0$  on the {\em Calkin
			algebra} $ \mathcal{B}(\mathcal{H})/\mathscr{K}(\mathcal{H}) $. Then
		$\xi:\mathscr{B}(\mathcal{H})\to\mathbb{C} $ defined as $ \xi= \xi
		_0\circ \pi $ is  a {\em non-normal} state on
		$\mathscr{B}(\mathcal{H}).$ Let $ \mbox{Q}\in
		\mathscr{B}(\mathcal{H})$ be a non-zero finite rank projection.
		Consider the pure CP map, $ \varphi
		:\mathscr{B}(\mathcal{H})\rightarrow\mathscr{B}(\mathcal{H})$
		defined by $$\varphi(\mbox{X})=\mbox{Q}\mbox{X}\mbox{Q},\forall
		\mbox{X} \in \mathscr{B}(\mathcal{H}).$$  Let
		$\psi:\mathscr{B}(\mathcal{H})\rightarrow\mathscr{B}(\mathcal{H})$
		be the CP map given by $$\psi(\mbox{X})=\frac12 \varphi(\mbox{X})
		+\frac12\xi(\mbox{X})\mbox{Q}.$$ Clearly, $
		\varphi(I)=\psi(I)=\mbox{Q} $ and $
		\varphi(\mbox{X}) - \psi(\mbox{X}) \in \mathcal{N}_{\xi}, \forall
		\mbox{X}\in\mathscr{B}(\mathcal{H})$. That is, $\varphi
		\underset{\xi}{=} \psi.$  But $  \varphi(\mbox{Q}) \neq
		\psi(\mbox{Q}), $ which implies that $ \varphi \neq \psi.$
	\end{example}

	\section*{Acknowledgments}
	Bhat gratefully acknowledges funding from  SERB (India) through JC
	Bose Fellowship No.
	
	\noindent 
	JBR/2021/000024. Chongdar thanks the  Indian
	Statistical Institute for providing research 
	fellowship.

\end{document}